\documentclass[11 pt]{amsart}

\usepackage{amsfonts, amssymb, amsmath, amsthm,  
latexsym, enumerate, epsfig, color, mathrsfs, 
fancyhdr, pdfsync, eurosym, endnotes, stmaryrd}
\usepackage[all]{xy}
\usepackage[left=2.5cm,top=2.3cm,right=2.5cm,bottom=2.3cm,bindingoffset=0.5cm]{geometry}
 \usepackage[english]{babel} 
\usepackage{amscd, url, tikz}

\def\A{{\mathbb A}} 
\def\C{{\mathbb C}} 
\def\P{{\mathbb P}} 
\def\Q{{\mathbb Q}} 
\def\R{{\mathbb R}}

\def\Z{{\mathbb Z}}

\def\CC{{\mathbb C}}

\def\Qcal{{\mathcal Q}}

\newcommand{\PPP}{\mathbb P}

\newcommand{\Ocal}{\mathcal O}

\newcommand{\Hcal}{\mathcal H}

\newcommand{\Jac}{\operatorname{Jac}}

\newcommand{\al}{\alpha}
\newcommand{\be}{\beta}
\newcommand{\ga}{\gamma}

\newcommand{\la}{\lambda}

\newcommand{\age}{\operatorname{age}}

\def\pmmu{{\pmb \mu}}

\def\CR{\operatorname{CR}}

\def\CY{\operatorname{CY}}
\def\LG{\operatorname{LG}}

\def\ol{\overline}

\def\wt{\widetilde}

\newtheorem{thm}{Theorem}[section]
\newtheorem{prop}[thm]{Proposition}
\newtheorem{lemma}[thm]{Lemma}
\newtheorem{cor}[thm]{Corollary}

\theoremstyle{remark} 
\newtheorem{rmk}[thm]{Remark}
\newtheorem{defn}[thm]{Definition}

\def\age{\mathop{\rm age}\nolimits}
\def\exp{\mathop{\rm exp}\nolimits}

\def\wbold{{\pmb w}}
\def\dbold{{\pmb d}}
\def\xbold{{\pmb x}}
\def\Wbar{{\overline W}}

\newcommand\eqnrf[1]{(\ref{#1})}

\begin{document}
\title{The Hybrid Landau--Ginzburg Models of Calabi--Yau Complete Intersections}
\author[A.~Chiodo]{Alessandro Chiodo}
\address{IMJ-PRG 
UMR7586, CNRS, UPMC, case 247, 4 Place Jussieu, 75252 Paris cedex 5, France}
\email{\tt{Alessandro.Chiodo@imj-prg.fr}}

\author[J.~Nagel]{Jan Nagel}
\address{IMB UMR5584, CNRS, Univ. Bourgogne Franche-Comt\'e, 21000 Dijon, France}
\email{\tt{Johannes.Nagel@u-bourgogne.fr}}

\begin{abstract}
We observe that the  state space 
of Landau--Ginzburg isolated singularities is simply a special case of 
Chen--Ruan orbifold cohomology relative to the generic fibre of the potential. 
This leads to the definition of 
the cohomology of hybrid Landau--Ginzburg models and its
identification 
via an explicit isomorphism to 
the cohomology of 
Calabi--Yau complete intersections inside weighted 
projective spaces. The combinatorial method used in the case of hypersurfaces 
proven by the first named author in collaboration with Ruan is streamlined 
and generalised after an orbifold
version of the 
Thom isomorphism and of the Tate twist.
\end{abstract}
\maketitle

\section{Introduction}
Landau--Ginzburg models play a central role in mirror symmetry. 
The simplest examples 
are the following isolated singularities: weighted-homogeneous polynomials
$W\colon \CC^n\to \CC$ with smooth fibres outside the origin and a single
isolated singularity in the special fibre over the origin. 
They played a key role in the early development of quantum cohomology, giving some of the first
examples of Frobenius manifolds and allowing to verify mirror symmetry 
statements. 
They have recently acquired a new status with 
the definition of their quantum invariants by Fan, Jarvis and Ruan \cite{FJRW} (FJRW theory) 
and later by Polishchuk and Vaintrob \cite{PV} 
(cohomological 
field theory of matrix factorisations). 
These advances make it possible to obtain equivalences between 
Landau--Ginzburg models and Calabi--Yau models without necessarily 
passing 
through mirror symmetry \cite{CIR}.
It is now also possible 
to state mirror symmetry between Landau--Ginzburg models both at the level
of  cohomological 
invariants and at the level of quantum cohomology invariants (see \cite{FJRW, Krawitz, BorisovBH}). There are many consequences: for instance 
the absence of any  
Calabi--Yau conditions in mirror statements can be effectively 
used (see for instance  \cite{FJRW, Krawitz});  
furthermore, the computational power appears higher:
We refer to approaches beyond concavity and 
in higher genus, see 
 \cite{FJRW, Gue, HLSW, CRZ}).

\medskip 

This paper starts from the definition by Fan, Jarvis and Ruan  \cite{FJRW}
of the state space
$\Hcal(W)$ of the Landau--Ginzburg model $W\colon [\CC^n/\pmmu_d]\to \CC$,
where $W$ is a potential with isolated singularity and $\pmmu_d$ is the 
corresponding monodromy action
$$\Hcal(W):=\bigoplus\nolimits_{g\in \pmmu_d} \mathcal (\Qcal_{W_g})^{\pmmu_d}.$$
Here $\Qcal_f$ stands for the Jacobian ring, or chiral algebra, attached to a potential $f$ 
\begin{equation}\label{eq:chiral}
\mathcal Q_f= dx_1\wedge\dots\wedge dx_n\otimes 
\Jac(f)=dx_1\wedge\dots\wedge dx_n\otimes\CC[x_1,\dots,x_n]/(\partial_1 f,\dots, \partial_n f),\end{equation}
and $W_g$ is the restriction of $W$ to the $g$-fixed space $\CC^{n_g}$ 
in $\CC^n$. 
As pointed out by the authors of \cite{FJRW}, the crucial ingredient $\Qcal$
and the orbifolding procedure 
appear already in the definition of the states of Landau--Ginzburg models
in Intriligator--Vafa \cite{IV} and in 
Kaufmann's work \cite{Kauf1},\cite{Kauf2},\cite{Kauf3}
in the definition of $B$-model invariants.
This state space is  bi-graded by applying to each term 
$(Q_{W_g})^{\pmmu_d}$ an
explicit formula: the somewhat \emph{ad hoc} bigrading
\begin{equation}\label{fakeTate}(n_g-\deg(f),\deg(f))+(\age(g),\age(g))-\left(\textstyle{\sum\nolimits_i \frac{w_i}d,\sum\nolimits_i \frac{w_i}d}\right),\end{equation}
 where $\age$ is the additive morphism 
$R\pmb \mu_r\to \Q$ mapping $k\in (\pmmu_r)^\vee=\Z/r$ with $0\le k<r$ to $k/r$.\\

\noindent{\emph{The cohomological LG/CY correspondence for CICY.}
In this paper we recognise 
that $\Hcal(W)$ is nothing but an occurrence 
of relative orbifold Chen--Ruan cohomology: for a Calabi--Yau hypersurface
$X_W=V(W)$ we consider the cohomological Landau--Ginzburg model
$$\Hcal^{p,q}(W)=H_{\CR}^{p+1,q+1}([\CC^n/\pmmu_d],W^{-1}(t_0);\CC)$$
for any $t_0\in \CC^\times$.
Note that $[\CC^n/\pmmu_d]$ is $\Ocal(-\pmb w):=\Ocal(-w_1)\oplus\dots\oplus \Ocal(-w_r)$ on 
the zero-dimensional weighted projective space $\PPP(d)$.
Then, when $W_1,\dots,W_r$ define a 
smooth complete intersection inside $\PPP(\pmb w)$, we set 
\begin{equation}
\label{hybrid LG statespace}
\Hcal^{p,q}(W_1,\dots,W_n):=H_{\CR}^{p+r,q+r}(\Ocal_{\PPP(\pmb d)}(-\pmb w); {\overline{W}}^{-1}(t_0);\CC), 
\end{equation}
where $\overline{W}$ may be regarded as a $\CC$-valued function, 
via the Cayley trick $\overline{W}=\sum_{i=1}^r p_iW_i(\pmb x)$ (see for instance 
\cite[5.4]{WiLGCY}).
In physics literature,  it is common to 
refer to the coordinates of $\PPP(\pmb d)$ as massive or 
ground coordinates and 
to the coordinates along the fibres of $\Ocal(-\pmb w)$ as 
massless coordinates, and to refer to $\CC$-valued potentials 
such as $\overline{W}$ effectively involving both types of variables as  
\emph{hybrid} (see, for instance, \cite[5.4]{WiLGCY} and \cite{HHP}). 

Finally, by a variation of stability condition one gets the 
counterpart of this hybrid model: 
the morphism 
$\overline{W}\colon 
\Ocal_{\PPP(\pmb w)}(-\pmb d)\to \CC$ which, 
after the same relative cohomology computation as above,
turns out to be isomorphic to the cohomology of the complete
Calabi--Yau intersection $V(W_1,\dots,W_r)$ inside $\PPP(\pmb w)$
 (see Prop.~\ref{prop:Tate} with $G=1$). The two LG and CY models are related by
 the following theorem.
We refer to Theorem \ref{thm} and Corollary 
\ref{cor} for the complete statement involving group quotients of complete intersections.

 \bigskip
 \noindent\textbf{Theorem (cohomological LG/CY correspondence for CICY).}
 \emph{For any smooth 
 complete intersection $X_{\overline{W}}=
 V(W_1,W_2,\dots,W_r)$ of Calabi--Yau type 
 inside the weighted projective space $\P(\wbold)$ we have 
 }
 $$\Hcal^{p,q}(W_1,\dots,W_r)\cong H^{p,q}_{\CR}(X_{W};\Q).$$
 
 \smallskip

\begin{rmk}
A related result by Libgober \cite{Libgober}  
proves
the invariance of the elliptic genus in a more general setup which includes complete intersections and several other GIT quotients 
(we recall that, after specialisation, 
the elliptic genus gives a combination of Hodge numbers, see 
\emph{e.g.} \cite{BL0}).
It is also worth-while pointing out how, in 
Libgober's treatment, Calabi--Yau models, Landau--Ginzburg models and their hybrid versions 
are precisely characterised in geometric terms, see \cite[Defn.~2.3]{Libgober}. 
\end{rmk}
  
Here, our main focus is how the Landau--Ginzburg/Calabi--Yau correspondence follows 
from several properties 
of Chen--Ruan cohomology of independent interest. We illustrate it in the rest of the 
introduction.
\\

\noindent {\emph{Thom isomorphism in Chen--Ruan cohomology and the Tate twist.}}
 The LG/CY statement, and the more general 
  Thm.~\ref{thm} including group actions, 
 relies on a version 
 of the Thom isomorphism  in ordinary cohomology
 \begin{equation}\label{eq:Thom}
 H^*(X;\CC)=H^*(\P(\wbold),\P(\wbold)\setminus X; \CC)\ (r),\end{equation} 
 where, by the Tate twist ``$(r)$'', the isomorphism preserves the
 bi-degree.
Notice that, due to its age-shifted degree, 
Chen--Ruan orbifold cohomology does not satisfy \eqref{eq:Thom} in 
the above form. However, by replacing $\P(\wbold)$ 
with the total space of the 
vector bundle $\Ocal(-d_1)\oplus \dots\oplus \Ocal(-d_r)$
we get a Thom isomorphism statement (see Prop.~\ref{prop:Tate} for the general statement).

\medskip 

\noindent \textbf {Proposition (Thom isomorphism in orbifold cohomology).} 
\emph{We have a canonical isomorphism 
$$H_{\CR}^{p-r,q-r}\left(X_W\right)\cong 
H^{p,q}_{\CR}\left(\Ocal_{\wbold}(-\dbold),\overline{W}^{-1}(t_0)\right).$$}
\smallskip

This is interesting in its own right
and explains, the \emph{ad hoc} bi-degree shift \eqref{fakeTate} in a geometric 
way: in Calabi--Yau cases it is merely the Tate twist in the Thom isomorphism.\footnote{This interpretation of the overall shift-by-total-charge $\textstyle{\sum\nolimits_i \frac{w_i}d}$ in
\eqref{fakeTate}, 
may lead to the correct interpretation of the age shift too;
it may be worth-while to point out here Denef--Loeser's \cite{DL} 
view on the age grading as a representation-theoretic analogue of the weight in Hodge theory, 
see Miles Reid \cite[\S2]{Reid}.}

In the same spirit, let us point out how the present description of 
the state space clarifies in classical geometric terms the dichotomy between 
narrow and broad 
sectors, which, following \cite{FJRW}, depends 
on the condition that the space of $g$-fixed points is 
or is not reduced to the origin. This dichotomy does not generalise as such 
in the complete intersection case. However, in our relative cohomology picture
broad states are simply those that lie 
in the kernel of the natural morphism $H^*_{\CR}(\Ocal(-\pmb w),\overline{W}^{-1}(t_0))\to 
H^*_{\CR}(\Ocal(-\pmb w))$. In this way ``broad'' is nothing but a 
generalisation of the classical notion of primitive.

 We also feel that 
the generalisation of Gysin sequences and Thom isomorphisms in Chen--Ruan 
orbifold cohomology is worth pursuing further 
in view of a well-behaved setup for an orbifold cohomology theory. \\

\noindent {\emph{An algebraic model for  the quantum theory of hybrid models.}}
We finally notice how, using the Thom isomorphism, 
the theorem stated above follows from a bi-degree preserving 
isomorphism 
$$H^*_{\CR}(\Ocal_{\P(\wbold)}(-\pmb d);\CC)\cong 
H^*_{\CR}(\Ocal_{\P(\dbold)}(-\pmb w);\CC).$$
This isomorphism 
 holds by $K$-equivalence but can be expressed directly via a
combinatorial argument which, with respect to \cite{ChiR}, is 
greatly simplified.

 By relying on this identification and on the classical descriptions of 
 the primitive cohomology of complete intersections \cite{Dimca, Na},
 we derive 
 an algebraic model of the cohomology of Calabi--Yau complete 
 intersection in terms of 
 the Jacobian 
 ring; see Theorem \ref{thm:FJRW}. This type of algebraic presentation of 
 cohomology classes has been proven crucial on many occasions in the formulation and proof of 
 mirror symmetry statements.
  
 The cohomology groups $\Hcal(W_1,\dots,W_n)$ 
 developed in this paper, have already been used 
 to provide a quantum
 correspondence between Calabi--Yau complete intersections 
 inside projective spaces 
 and hybrid Landau--Ginzburg models, see Clader \cite{Clader}. Recently, Fan, Jarvis and Ruan 
 provided a generalisation of their theory based on the 
 state space $\Hcal(W_1,\dots, W_n)$, see \cite{FJRbis}.
 Clader establishes a  correspondence between the quantum D-modules attached to 
 complete intersection of CY type within $\PPP(\CC^n)$ and the corresponding hybrid LG models. 
 Her correspondence depends on the choice of an analytic
  continuation
 which should mirror the parallel transport 
along the Gau\ss--Manin connection on the B-side. 
In this way, all
isomorphisms arising at the cohomological level from these quantum
correspondences
should depend on the path chosen for the parallel transport. 
In this perspective it is interesting how our proof of Theorem \ref{thm} 
provides a combinatorial method for writing an isomorphism 
explicitly. 
Therefore, our isomorphism could provide a useful tool to 
study the monodromy operators in the future. \\

\noindent {\emph{Structure of the paper.}} We recall the setup in \S2, develop the 
Calabi--Yau side in \S3, the Landau--Ginzburg side in \S4, and their correspondence in \S5.\\

\noindent {\emph{Aknowledgements.}} The first named author is grateful to the IMB, Dijon, for
the hospitality during the finalisation of this work. We are extremely 
grateful to Anatoly Libgober for his 
comments and suggestions.

\section{Setup}
We consider a complete intersection of $r<n$ hypersurfaces in the weighted projective space $\P(\wbold):=\P(w_1,\ldots,w_n)$.
Let $W_1,\ldots,W_r$ be quasi-homogeneous polynomials of degree
$d_1,\ldots,d_r$ in the variables $x_1,\dots, x_n$ 
of coprime weights $w_1,\dots,w_n$. 
We assume that $X_W$ is a non-degenerate complete intersection,
\emph{i.e.}, 
\begin{enumerate}
\item[(i)] the  choice of weights $w_1,\dots, w_n$ is unique;
\item[(ii)] $\{W_1=0,\ldots,W_r=0\}\subset\C^n$ is smooth outside
the origin.
\end{enumerate}
Assumption (ii) holds if and only if the Jacobian matrix
 $J = ({\partial_j W_i}(p))_{i,j}$ has rank $r$ for any 
$p\in\{\xbold \in \C^n\mid W_1=\dots=W_r=0\}\setminus\{\pmb 0\}$. 
Then, inside the Deligne--Mumford stack $\P(\wbold)$, we 
get a smooth complete intersection
$$
X_W = \{W_1=\ldots = W_r = 0\}\subset\P(\wbold).
$$

We further assume that $X_W$ is of Calabi--Yau type in the following
sense: we impose that the canonical sheaf is trivial, in other words
\begin{enumerate}
\item[(iii)] $\sum_{i=1}^r d_i = \sum_{j=1}^n w_j$.
\end{enumerate}

We consider a group 
$\Gamma\subset (\C^\times)^n$ of diagonal 
matrices $\ga$ whose diagonal 
entries are of the form $(\ga_1,\dots, \ga_n)
=\al \ol \la:=(\al_1\lambda^{w_1},\dots,\al_n\la^{w_n})$
with 
$W_i(\al_1x_1,\dots,\al_nx_n)=W_i(x_1,\dots,x_n),$ 
\text{all $i$} and $\la\in \C^\times$. We assume that 
$\Gamma$ contains the torus $\{\ol \la=(\la^{w_1},\dots,\la^{w_n})\mid \la \in \C^\times\}$; 
this is automatically its identity component $\Gamma^0$, and we write $G$ for 
the group of connected components $\Gamma/\Gamma^0$.
By (i), $G$ is finite (see for instance \cite[Lem.~2.18]{FJRW}). 
The quotient stack 
$[U/\Gamma]$, with $U=\{W_1 = W_2=\dots=W_r\}_{\C^n} \setminus \{\pmb 0\}$, 
is 
a smooth Deligne--Mumford stack and,
following Romagny's treatment \cite[Rem.~2.4]{Ro} 
of actions on stacks, it is canonically equivalent to the $2$-stack 
$$[X_W/G]=[U/\Gamma].$$
We introduce for all $\ga\in \Gamma$ the 
following notation
\begin{align*}
\C^n_\ga:=\{\xbold\in \C^n\mid \ga\cdot \xbold=\xbold\},\qquad 
n_\ga:=\dim\C^n_\ga,\qquad
W_{i,\ga}:=W_i\!\mid_{\C^n_\ga}.
\end{align*}

\begin{lemma}\label{lem:normalnotnormal} Consider 
$\ga=\al\ol\la:=\al(\lambda^{w_1},\dots,\lambda^{w_n})$
in $\Gamma$. 
We have $\ga^*W_i=W_i(\ga\cdot \xbold)=\la^{d_i}W_i$.
Therefore $W_{i,\ga}=0$ if $\la$ is not a $d_i$-th root of unity. Otherwise, for $\la\in \pmmu_{d_i}$,
$W_i$ is fixed by $\ga$, and the following conditions are satisfied: 
$W_{i,\ga}\neq 0$ and its partial derivatives coincide with
those of $W_i$ on $\C^n_\ga$.
\end{lemma}
\begin{proof}
Indeed, if $\la \in \pmmu_{d_i}$, then $W_{i}$ 
is of the form 
$W_{i,\ga}+f$, where  
$f$ belongs to the maximal ideal $\mathfrak m_\ga$ 
spanned by the 
variables $x_j$ not fixed by $\ga$ 
(\emph{i.e.} $\ga_j
=\al_j\la^{w_j}\neq 1$). In fact, since the action is diagonal, 
we have $f\in (\mathfrak m_\ga)^2$, 
which implies that, after restriction to 
$\C^n_\ga$,
the partial derivatives $\partial_j W_i$ 
and $\partial_j W_{i,\gamma}$ coincide for all $j=1,\dots,n$. In 
particular, by (ii), we have $W_{i,\ga}\neq 0$.
\end{proof}

\section{Calabi--Yau side.}
On the Calabi--Yau side of the correspondence, 
we consider the Chen--Ruan  orbifold cohomology of 
$[X_W/G]=[U/\Gamma]$. 
Since the groups acting are Abelian, we can express this explicitly 
as a direct sum of $\Gamma$-invariant parts of cohomology groups  
$$H^{p,q}_{\CR}([X_W/G];\Q)=\bigoplus_{\gamma\in \Gamma} 
H^{p-a_\ga,q-a_\ga}([X_{W,\ga}]; \Q)^\Gamma,$$
where $a_\ga$ is the age
of the action of $\ga$ on the tangent bundle of $X_W$ 
at a point of $X_{W,\ga}$
and $X_{W,\ga}$ is the 
quotient stack 
$$[\{W_{1,\ga}=\dots=W_{r,\ga}=0\}_{\C^n_\ga}/\C^\times]
\hookrightarrow \P(\wbold)_\ga,$$
where $\P(\wbold)_\ga$ is the weighted projective space spanned
by the $\ga$-fixed coordinates $x_j$.
\begin{rmk}We remark that the 
$$r_\ga:=\#\{W_{i,\ga}\mid W_{i,\ga}\neq 0\}\in [0,r]$$
restricted polynomials that do not vanish on $\C^n_\ga$ 
define a smooth complete intersection 
of $r_\ga$ hypersurfaces; indeed, 
the Jacobian matrix of the complete intersection in $\C^n_\ga$ 
has $r_\ga$ rows which, as pointed out above,
coincide with the restrictions of the corresponding rows in 
the Jacobian matrix of $X_W$ and, by (ii), form a matrix of rank $r_\ga$
at the points of $\C^n_\ga$ where $W_{i,\ga}=0$ for all $i$.
\end{rmk}
\begin{rmk}
All stacks considered in this paper are Deligne--Mumford stacks
which are global quotients $[Y/H]$. 
In these cases when we consider 
ordinary cohomology with coefficients in $\Q$ of 
the stack $[Y/H]$ and of the coarse scheme $Y/H$ we get the 
same result via a canonical isomorphism. 
The groups $H$ considered in this paper are 
at worse finite extensions of tori operating with finite stabilisers. 
Sometimes we consider non-faithful actions: 
$H\mapsto \mathrm{Aut}(Y)$ with 
nontrivial kernel $K$. In these cases it is important 
to notice that, whereas for ordinary cohomology
it does not matter if we consider $[Y/H]$ or $[Y/(H/K)]$, when it comes 
to Chen--Ruan orbifold cohomology these two stacks yield different cohomology groups, 
as the formula recalled above  obviously shows. 

As a consequence of these considerations we point out that in the above 
formula for Chen--Ruan orbifold cohomology $H^{p,q}_{\CR}([X_W/G];\Q)$, 
we can replace on the right hand side  
$H^*([X_{W,\ga}]; \Q)^\Gamma$ by $H^*([X_{W,\ga}]; \Q)^G$; indeed 
the identity component $\Gamma^0\cong \CC^\times$ acts trivially on $X_W$  and 
$\P(\wbold)$ and 
on each sector $X_{W,\ga}$ and $\P(\wbold)_\ga$. 
\end{rmk}

``Landau--Ginzburg models'' is an expression often used for 
$\C$-valued functions
defined on vector spaces or, more generally on vector bundles, 
as in this paper. In this context, the $\C$-valued functions are often called 
superpotentials.
In order to relate the Calabi--Yau complete intersection $X_W$ 
to a Landau--Ginzburg model, 
we rephrase its cohomology in 
terms of relative cohomology of a vector bundle.

For any $i=1,\dots, r$,
we can define a character $\chi_i\colon \Gamma\to \CC^\times$
by mapping $\al\ol\lambda\mapsto \la^{d_i}$; 
indeed, if $\al$ and $\be\in (\C^\times)^n$ preserve all polynomials 
$W_i$ and 
satisfy $\al\ol \la=\ga=\be\ol \mu\in \Gamma$ 
for some $\la,\mu\in \C^\times$, then
we have $\la^{d_i}=(\ga^*W_i)/W_i=\mu^{d_i}$ all $i$.
\begin{rmk}
We notice that $r_\ga$ equals $\#\{i\mid \chi_i(\ga)=0\}$.
See Lemma \ref{lem:normalnotnormal}.\end{rmk}
In this way,
we can define a $\Gamma$-action on $\C^{n+r}$ by
$$\ga\cdot(x_1,\ldots,x_n,p_1,\ldots,p_r) =
(\ga_1x_1,\ldots,\ga_nx_n,\chi_1(\ga)^{-1}p_1,\ldots,\chi_r^{-1}(\ga)p_r),
$$
or, more explicitly, by 
$
\al\ol \lambda\cdot(\xbold,{\pmb p}) =
(\al_1\lambda^{w_1}x_1,\ldots,\al_n\lambda^{w_n}x_n,\lambda^{-d_1}p_1,\ldots,\lambda^{-d_r}p_r).
$
We  consider the $\Gamma$-invariant $\C$-valued function 
\begin{align*}
\Wbar\colon \C^{n+r}&\to \C\\
(\pmb x,\pmb p)&\mapsto p_1W_1(\pmb x)+\ldots + p_rW_r(\pmb x).
\end{align*}
The fibre $M=\{\Wbar=t_0\}$ over 
any point $t_0\neq 0$ in $\C$ is smooth 
and its cohomology does not 
depend on the choice of $t_0\in \C^\times$.

We consider 
\begin{eqnarray*}
U_{\CY} & = & \C^n\setminus \{0\}\times \C^r 
\end{eqnarray*}
and the corresponding vector bundle $[U_{\CY}/\C^\times]$
$$\Ocal_\wbold(-\dbold):=\bigoplus_{i=1}^r \Ocal_{\P(\wbold)}(-d_i),$$
which we identify with the total space over $\P(\wbold)$.
If we consider the $\Gamma$-action on 
$\C^{n+r}$, we study the corresponding (total space of the) 
vector bundle $[\Ocal_\wbold(-\dbold)/ G]\to [\P(\wbold)/ G]$.

The function $\ol W$ descends to a $\C$-valued function $W_{\CY}$ 
on $[\Ocal_\wbold(-\dbold)/G]$ 
$$\ol W_{\CY}\colon [\Ocal_\wbold(-\dbold)/G]\to \C.$$ 

\begin{prop}\label{prop:Tate}
We have a canonical isomorphism 
$$H_{\CR}^{p-r,q-r}\left([X_W/G]\right)\cong 
H^{p,q}_{\CR}\left([\Ocal_{\wbold}(-\dbold)/ G],[F/G]\right),$$
where $F$ 
is the quotient stack $[M/\C^\times]$ inside 
$\Ocal_\wbold(-\dbold)$. \end{prop}

\begin{rmk}
$F$ is the generic  fibre of the morphism 
$\Ocal_\wbold(-\dbold)\to \Ocal\to \C$.
\end{rmk}
\begin{proof}
 On both sides we have a 
direct sum over $\ga\in\Gamma$
involving $\Gamma$-invariant classes of $X_{W,\ga}$ and 
of the pair 
$(\Ocal_{\wbold}(-\dbold)_\ga,F_\ga)$.
We notice that $\ga$ acts with age 
$a_\ga$ on the tangent bundle of $X_W$ 
and with age $a_\ga+(r-r_\ga)$ on the tangent bundle of 
$\Ocal_{\wbold}(-\dbold)$.
Indeed, using the chain of inclusions $X_W\subset\P(\wbold)\subset\Ocal_{\wbold}(-\dbold)$ (the last one being obtained by identifying $\P(\wbold)$ with the zero section) 
one readily shows that the normal bundle of 
$X_W$ inside $\Ocal_{\wbold}(-\dbold)$
is $\Ocal_{\wbold}(\dbold)\oplus \Ocal_{\wbold}(-\dbold)$.
For any fixed element $\ga\in \Gamma$, in view of age computations, 
we may ignore the 
$r_\ga$ summands of $\Ocal_{\wbold}(-\dbold)$ where 
$\ga$ acts trivially. 
Since, for a nontrivial character $\chi$, $\age(\chi)+\age(\chi^{-1})=1$
we obtain the desired age shift difference $r-r_\ga$.
 
Finally we need to check
\begin{equation}\label{eq:sectorclaim}H^{p-r_\ga,q-r_\ga}(X_{W,\ga})=
H^{p,q}(\Ocal_{\wbold}(-\dbold)_\ga,F_\ga),\end{equation}
which requires the following analysis of $F_\ga$.

\begin{lemma}
\label{lemma: cohom Milnor fiber} 
For any $\ga\in \Gamma$, if $r_\ga<n_\ga$, 
then $X_{W,\ga}$ is nondegenerate and we have 
$$
H^k(F_\ga) = \left\{
\begin{array}{cc}
H^k(\P(\wbold)_\ga) & 0\le k\le 2r_\ga-2 \\
H^{n_\ga-r_\ga-1}_{\rm pr}(X_{W,\ga}) & k=n_\ga+r_\ga-2 \\
0 & {\rm\ otherwise}
\end{array}
\right.
$$
If $r_\ga\ge n_\ga$ then $H^k(F_\ga)\cong H^k(\P(\wbold)_\ga)$ 
for all $k$.
\end{lemma}

\begin{proof} Set $X=X_{W,\ga}$ and $F=F_\ga$. By abuse of notation 
we write $\P(\wbold)_\ga$ as $\P(\wbold)$, $n_\ga$ and $r_\ga$ as $n$ and $r$. 
Consider the projection map
$\pi\colon F\to\P(\wbold)$. Since we may regard $F$ as
the variety defined by the equation $\sum_i
p_iW_i(x) = 1$, the fiber $\pi^{-1}(x)$ is empty if $x\in
X$ and an affine hyperplane isomorphic to $\A^{r-1}$ if $x\notin
X$. More precisely, $\pi$ is a locally trivial fibration over $\P(\wbold)\setminus X$ with fiber $\A^{r-1}$ (trivialize over the open subsets 
$U_i = \{\xbold\in\P(\wbold)|W_i(\xbold)\ne 0\}$). Hence $H^*(F)\cong H^*(\P(\wbold)\setminus X)$. As $X$ is a $\Q$-homology
manifold we have an isomorphism of Hodge structures
$$
H^k(\P(\wbold),\P(\wbold)\setminus X;\Q)\cong H^{k-2r}(X;\Q)(-r).
$$
Hence the long exact sequence of relative cohomology for the pair $(\P(\wbold),\P(\wbold)\setminus X)$ can be rewritten as 
$$
H^k(\P(\wbold))\to H^k(\P(\wbold)\setminus X)\to H^{k-2r+1}(X)\xrightarrow{i_*} H^{k+1}(\P(\wbold))
$$
where cohomology is with $\Q$-coefficients and the Gysin map $i_*$ is Poincar\'e dual to the map 
$$
i_*:H_{2n-k-3}(X)\to H_{2n-k-3}(\P(\wbold)).
$$
The result then follows from the Lefschetz hyperplane theorem.
\end{proof}
We now continue the proof of Proposition \ref{prop:Tate} If $r_\ga<n_\ga$ and $X_W$ is nondegenerate then
\begin{equation}
\label{eq: relative cohomology}
H^k(\Ocal_{\wbold}(-\dbold)_\ga,F_\ga)\cong H^{k-2r}(X_\ga)(-r).
\end{equation}
because the pullback via 
$\pi\colon \Ocal_{\wbold}(-\dbold)\to \P(\wbold)$ 
induces an isomorphism 
$$
H^k(\P(\wbold)_\ga,\P(\wbold)_\ga\setminus X_{W,\ga}))
\xrightarrow{\sim} H^k(\Ocal_{\wbold}(-d)_\ga,F_\ga)
$$
and $H^k(\P(\wbold)_\ga,\P(\wbold)_\ga\setminus X_{W,\ga}))
\cong H^{k-2r}(X_{W,\ga})(-r)$ as we have seen above. If 
$r_\ga\ge n_\ga$ the same results hold, with $X = \varnothing$.
\end{proof}

\section{Hybrid Landau--Ginzburg side}
On the other side 
we consider the open set 
$$U_{\LG}=\C^n\times (\C^r\setminus\{\pmb 0\})$$
and the (total space of the) vector bundle \begin{eqnarray*}
{[{U_{\LG}/\C^\times}]}=\bigoplus_{j=1}^n
\Ocal_{\P(\dbold)}(-w_i),
\end{eqnarray*}
which we simply denote by $\Ocal_\dbold(-\wbold)$.
We notice that the Milnor fiber 
$M = \Wbar^{-1}(t_0)$ for $t_0\neq 0$ is contained in both
$U_{CY}$ and $U_{LG}$.
Its cohomology does not depend on $t_0$ and, after modding 
out the $\C^\times$-action, has been completely described in 
Lemma \ref{lemma: cohom Milnor fiber}.

We recall that $\Gamma$ is a group of diagonal symmetries acting on 
$\C^n$; as in the previous section we extend its action 
to $\C^n\times\C^r$.
We now consider the {\em hybrid LG model}
$(\Ocal_{\dbold}(-\wbold),\Wbar)$ with superpotential
$$
\Wbar\colon[\Ocal_{\dbold}(-\wbold)/G]\to\C. $$
As above, the quotient stack $[\Ocal_{\dbold}(-\wbold)/G]$ may be 
presented as 
$[U_{\LG}/\Gamma]$ and may be regarded as a 
vector bundle over $[\P(\dbold)/G]$.

\begin{defn}\label{defn:states}
The generalized state space of the hybrid Landau--Ginzburg model $(\Ocal_{\dbold}(-\wbold),\Wbar)$ is 
$$
\Hcal^{p,q}_{\Gamma}(W_1,\ldots,W_r):
=
H^{p+r,q+r}_{\CR}([\Ocal_{\dbold}(-\wbold)/G],[F/G]),
$$
where $G$ is the component group $\Gamma/\Gamma^0$.
\end{defn}

\begin{rmk} 
We point out that the space $\Hcal^{p,q}(W_1,\ldots,W_r)$ defined in \eqref{hybrid LG statespace} coincides with the
above definition for $\Gamma=\C^\times$. 
\end{rmk}

We can describe in greater detail the relative Chen--Ruan orbifold cohomology
$$H^{p,q}_{\CR}([\Ocal_{\P(\dbold)}(-\wbold)/G],[F/G])=
\bigoplus_{g\in G}
H^{p-a_g,q-a_g}([\Ocal_{\P(\dbold)}(-\wbold)/G]_g,[F/G]_g)$$
or, more explicitly, 
\begin{equation}
\label{eq:FJRW}
 H^{p,q}_{\CR}([U_{\LG}/\Gamma],[M/\Gamma])=
\bigoplus_{\ga\in \Gamma}
H^{p-a_\ga,q-a_\ga}(\Ocal_{\dbold_\ga}(-\wbold_\ga),F_\ga)^\Gamma,
\end{equation}
where $\dbold_\ga$ and $\wbold_\ga$ are the multi-indices obtained by suppressing the entries $i$ and $j$ for which $\ga_i\neq 1$, and 
$a_\ga$ is the age of the action of $\ga$ on $\C^n\times \C^r$.
We recall again that $H^*(\Ocal_{\dbold_\ga}(-\wbold_\ga),F_\ga)^\Gamma$
is just the group of $G$-invariant classes  
$H^*(\Ocal_{\dbold_\ga}(-\wbold_\ga),F_\ga)^G$. 

This allows us to give an explicit presentation of FJRW cohomology
as a direct sum of two types of terms reflecting the 
hybrid nature of our Landau--Ginzburg models: 
a chiral algebra term computing the primitive cohomology via
the Jacobian ring, and a term computing the fixed cohomology. 

The analogue of the chiral algebra
of the Jacobian ring 
of an isolated singularity is defined as follows.
For any smooth 
complete intersection of $\wt r$ hypersurfaces $\{W_i=0\}$ 
(with  
$i=1,\dots, \wt r$) inside 
a weighted projective space of $\wt n>\wt r$ coordinates 
we consider the generalised chiral algebra
$$\mathcal Q_{\wt W}= 
dx_1\wedge\ldots\wedge dx_{\wt n}\wedge 
dp_1\wedge\ldots\wedge dp_{\wt r}\otimes
\CC[x_1,\dots,x_{\wt n},p_1,\dots,p_{\wt r}]/(\partial_{x_j}\wt W,
\partial_{p_i}\wt W),$$
where $\wt W=\sum_{i=1}^{\wt r}p_i W_i(\xbold)$. This is the Jacobian 
ring from \cite{Dimca} or \cite{Na} tensored with the 
top-degree form. Its 
$\CC^\times$-invariant part 
is finite dimensional.
For $\wt D= \wt n-\wt r-1$, by \cite[Thm. 7]{Dimca}, \cite[Thm. 2.16]{Na}, we have 
$$(\mathcal Q_{\wt W})^{\C^\times}\cong 
H^{\wt D}_{\rm pr}(X_{W,\ga})(-r).$$
If we assign bi-degree $(\wt D-k,k)$ to each term of 
$\mathcal Q_{\wt W}$ 
degree $k$ in the variables $p_1,\dots,p_{\wt r}$ we may regard 
the above isomorphism as a bi-degree preserving isomorphism. 
Note that in both papers cited above the result is stated 
for complete intersections in projective space; 
the same proof goes through in the case of weighted 
projective spaces \cite[Remark 18 (i)]{Dimca}.

Finally for any non-negative 
integer $\wt n$ and for any $\wt r\ge \wt n$ 
we need to consider the cohomology of $\P(\dbold)$ in degrees $0,\ldots,2r-2n-2$.
We write the bi-degree preserving identification 
$$dt(-\wt n) \otimes \CC[t]/
(t^{\wt r-\wt n})\cong 
\bigoplus_{k=0}^{\wt r- \wt n - 1} H^{2k}(\P(\dbold)).$$
where 
 $dt(-\wt n)\otimes t^k$ has bidegree $(\wt n,\wt n)+(k,k)$.

\begin{thm}
\label{thm:FJRW}
We have 
$$\Hcal^{*}_{\Gamma}(W_1,\ldots,W_r)=
\oplus_{\ga\in \Gamma} H_\ga(-a_\ga+r),$$
where $H_\ga$ with its double grading 
is given by 
$$H_\ga=\begin{cases} (\mathcal Q_{W_\ga})^\Gamma & 
\text{if }r_\ga<n_\ga,\\
dt (-n_\ga)\otimes \CC[t]/(t^{r_\ga-n_\ga}) 
& \text{if }r_\ga\ge n_\ga.\end{cases}$$
\end{thm}

The proof follows from the analysis of the relative cohomology of
$(\Ocal_{\P(\dbold_\ga)}(-\wbold_\ga)_\ga,F_\ga)$ for each sector. 
We assume $G=1$ for simplicity; as we saw above, nontrivial groups $G$ can 
be treated easily by the same argument. By abuse of notation we write
$\dbold,\wbold$ and $F$ instead of $\dbold_\ga,\wbold_\ga$ and $F_\ga$. The following lemma provides a complete picture.
\begin{lemma}
\label{lemma: relative cohom II}
If $r<n$ and $X$ is a nondegenerate 
complete intersection, 
then
$$
H^k(\Ocal_{\dbold}(-\wbold),F) = \left\{
\begin{array}{cc}
H^{n-r-1}_{\rm pr}(X)(-r) & {\rm\ if\ } k = n+r-2 \\
0 & {\rm\ if\ } k\ne n+r-2.
\end{array}
\right.
$$
If $r\ge n$ then 
$$
H^k(\Ocal_{\dbold}(-\wbold),F)\cong 
\left\{
\begin{array}{cc}
H^{k-2n}(\P(\dbold))(-n)& {\rm\ if\ } 2n\le k\le 2r-2 \\
0 & {\rm otherwise.}
\end{array}
\right.
$$
\end{lemma}

\begin{proof}
Let $Z_1\subset\Ocal_{\wbold}(-\dbold)$ and $Z_2\subset\Ocal_{\dbold}(-\wbold)$ be the
zero sections. Write
$$
\Ocal_{\wbold}(-\dbold)^{\times} = \Ocal_{\wbold}(-\dbold)\setminus Z_1,\ \ 
\Ocal_{\dbold}(-\wbold)^{\times} = \Ocal_{\dbold}(-\wbold)\setminus Z_2.
$$

The vector bundles $\Ocal_{\dbold}(-\wbold)$ and
$\Ocal_{\wbold}(-\dbold)$ contain a common open subset
$$
[U_{CY}\cap U_{LG}/\C^\times]\cong\Ocal_{\wbold}(-\dbold)^{\times}\cong\Ocal_{\dbold}(-\wbold)^{\times}. 
$$
The inclusions
$F\subset\Ocal_{\dbold}(-\wbold)^{\times}\subset\Ocal_{\dbold}(-\wbold)$ give an inclusion
of pairs
$$
(\Ocal_{\dbold}(-\wbold)^{\times},F)\subset(\Ocal_{\dbold}(-\wbold),F)\subset(\Ocal_{\dbold}(-\wbold),\Ocal_{\dbold}(-\wbold)^{\times}).
$$
We shall calculate the relative cohomology
$H^k(\Ocal_{\dbold}(-\wbold),F)$ using the exact sequence of relative cohomology
\begin{equation}
\label{exseq}
H^k(\Ocal_{\dbold}(-\wbold),\Ocal_{\dbold}(-\wbold)^{\times})\xrightarrow{} 
H^k(\Ocal_{\dbold}(-\wbold),F)\xrightarrow{} H^k(\Ocal_{\dbold}(-\wbold)^{\times},F).
\end{equation}
We have
$$
H^k(\Ocal_{\dbold}(-\wbold),\Ocal_{\dbold}(-\wbold)^{\times})=H^k_{Z_2}(\Ocal_{\dbold}(-\wbold))\cong
H^{k-2n}(Z_2)\cong H^{k-2n}(\P(\dbold)).
$$
To calculate $H^k(\Ocal_{\dbold}(-\wbold)^{\times},F)$, we use the isomorphism $H^k(\Ocal_{\dbold}(-\wbold)^{\times},F)\cong H^k(\Ocal_{\wbold}(-\dbold)^{\times},F)$ and the previous results for $\Ocal_{\wbold}(-\dbold)$. Using the long exact sequence
$$
\ldots\to H^k(\Ocal_{\wbold}(-\dbold),\Ocal_{\wbold}(-\dbold)^{\times})\to H^k(\Ocal_{\wbold}(-\dbold),F)\to H^k(\Ocal_{\wbold}(-\dbold)^{\times},F)\to\ldots
$$
coming from the inclusion of pairs 
$$
(\Ocal_{\wbold}(-\dbold)^{\times},F)\subset(\Ocal_{\wbold}(-\dbold),F)\subset(\Ocal_{\wbold}(-\dbold),\Ocal_{\wbold}(-\dbold)^{\times}),
$$
the isomorphisms 
\begin{eqnarray*}
H^k(\Ocal_{\wbold}(-\dbold),\Ocal_{\wbold}(-\dbold)^{\times})& \cong & H^{k-2r}(\P(\wbold)) \\
H^k(\Ocal_{\wbold}(-\dbold),F) & \cong & H^{k-2r}(X)\ (see\ (\ref{eq: relative cohomology}))
\end{eqnarray*}
and the five lemma, we obtain
$$
H^k(\Ocal_{\wbold}(-\dbold)^\times,F)\cong H^{k-2r+1}(\P(\wbold),X).
$$

\medskip
We now turn to the calculation of $H^k(\Ocal_{\dbold}(-\wbold),F)$ in the case $r<n$. If $k\ge
2r$ the exact sequence 
$$
H^{k-1}(F)\to H^k(\Ocal_{\dbold}(-\wbold),F)\to H^k(\Ocal_{\dbold}(-\wbold)) 
$$
shows that
$$
H^k(\Ocal_{\dbold}(-\wbold),F)\cong H^{k-1}(F) = \left\{
\begin{array}{cc}
H^{n-r-1}_{\rm pr}(X) & {\rm\ if\ } k=n+r-1 \\
0 & {\rm\ otherwise.}
\end{array}
\right.
$$
If $k\le 2r-1$ the exact sequence (\ref{exseq})
shows that
$H^k(\Ocal_{\dbold}(-\wbold),F) = 0$ since the terms $H^{k-2n}(\P(\dbold))$
and
$$ H^k(\Ocal_{\dbold}(-\wbold)^{\times},F) =
H^k(\Ocal_{\wbold}(-\dbold)^{\times},F)\cong H^{k-2r+1}(\P(\wbold),X)
$$
vanish (also for $k=2r-1$, since $H^0(\P(\wbold))\cong H^0(X)$).

Hence in the case $r<n$ we find
$$
H^k(\Ocal_{\dbold}(-\wbold),F) = \left\{
\begin{array}{cc}
H^{n-r-1}_{\rm pr}(X) & {\rm\ if\ } k=n+r-1 \\
0 & {\rm\ otherwise.}
\end{array}
\right.
$$

\medskip
We now consider the case $r\ge n$. The calculations 
go through as before, with $X = \varnothing$.
Using the exact sequence of relative cohomology for the pair $(\Ocal_{\dbold}(-\wbold),F)$  we find that
$H^k(\Ocal_{\dbold}(-\wbold),F) = 0$ if $k\ge 2r$.

Next we consider the case $k\le 2r-2$. The exact sequence
\eqnrf{exseq} then shows that $H^k(\Ocal_{\dbold}(-\wbold),F)\cong
H^{k-2n}(\P(\dbold))$, hence
$$
H^k(\Ocal_{\dbold}(-\wbold),F) = \left\{
\begin{array}{cc}
\Q & {\rm\ if\ } 2n\le k\le 2r-2, k\ {\rm even} \\
0 & {\rm\ otherwise.}
\end{array}
\right.
$$
(Note that if $r=n$ the formula says that $H^k(\Ocal_{\dbold}(-\wbold),F) =
0$ for all $k\le 2r-2$.)

For the remaining case $k=2r-1$ we use an Euler characteristic
calculation. The exact sequence \eqnrf{exseq}
shows that
$$
\sum_i (-1)^i (h^i(\Ocal_{\dbold}(-\wbold),\Ocal_{\dbold}(-\wbold)^{\times}) - h^i(\Ocal_{\dbold}(-\wbold),F)\\ + h^i(\Ocal_{\dbold}(-\wbold)^{\times},F)) =
0.
$$
By the previous calculations we have
\begin{eqnarray*}
\sum_i (-1)^i h^i(\Ocal_{\dbold}(-\wbold),\Ocal_{\dbold}(-\wbold)^{\times}) & = & \sum_i
(-1)^i h^{i-2n}(\P(\dbold)) = r \\
\sum_i (-1)^i h^i(\Ocal_{\dbold}(-\wbold),F) & = &
r-n-h^{2r-1}(\Ocal_{\dbold}(-\wbold),F) \\
\sum_i h^i(\Ocal_{\dbold}(-\wbold)^{\times},F) & = & \sum_i (-1)^i
h^{i-2r+1}(\P(\wbold)) = -n.
\end{eqnarray*}
Hence we find that $H^{2r-1}(\Ocal_{\dbold}(-\wbold),F) = 0$.

So in the case $r\ge n$ we obtain
$$
H^k(\Ocal_{\dbold}(-\wbold),F) = \left\{
\begin{array}{cc}
H^{k-2n}(\P(\dbold)) & {\rm\ if\ } 2n\le k\le 2r-2 \\
0 & {\rm\ otherwise.}
\end{array}
\right.
$$
\end{proof}

\begin{proof}[Proof of Theorem \ref{thm:FJRW}]
The claim follows immediately from the formula (\ref{eq:FJRW}) and Lemma \ref{lemma: relative cohom II}.
\end{proof}

\section{The correspondence}
\begin{thm}\label{thm}
Let $X_W = \{W_1=0,\ldots,W_r=0\}\subset\P(w_1,\ldots,w_n)$ be a
non-degenerate Calabi--Yau complete intersection. There exists an
explicit bidegree preseving isomorphism
$$
H^*_{\CR}([\Ocal_{\P(\dbold)}(-\wbold)/G])\cong 
H^*_{\CR}([\Ocal_{\P(\dbold)}(-\wbold)/G]).
$$
\end{thm}
As an immediate corollary we get the desired isomorphism.
\begin{cor}\label{cor}
There exists an explicit bidegree preserving isomorphism
$$
H^*_{\Gamma}(W_1,\ldots,W_r)\cong 
H^*_{\CR}([X_W/G]).
$$
\end{cor}

\begin{proof}[Proof of Theorem \ref{thm}]
Let us 
refer to the left hand side as the CY side, and to the right hand side as the LG side. 
By definition, both sides decompose as the direct sum over $\Gamma$
$${\rm CY}=\bigoplus_{\substack{{\ga \in \Gamma}}}\bigoplus_{0\le k<n_\ga} [{\rm CY}^k_\ga] \Q \qquad \text{and} \qquad 
{\rm LG}=\bigoplus_{\substack{{\ga \in \Gamma}} }
\bigoplus_{0\le k<r_\ga} [{\rm LG}^k_\ga] \Q,$$
where all terms $\CY_\ga^k$ and $\LG_\ga^k$ 
are elements of the form 
$H^k\cap \pmb 1_\ga$, where $H$ is an hyperplane section
and $\pmb 1_\ga$ is the fundamental class of the corresponding sector. The 
bidegree $(p,p)$ of such elements  
is $p=k-1+a_\ga$ by definition.

We shall identify the $\CY$ and the $\LG$ side for every connected component of $\Gamma$.
Each component is of the form $g\ol \lambda$ for a fixed diagonal symmetry 
of $W_1,\dots,W_n$. 
We express $g$ as the diagonal matrix whose $n+r$ entries 
on the diagonal are $(\exp(2\pi i a_1),\dots, 
\exp(2\pi i a_n),{\bf 1})$ with $a_i\in [0,1[$.
We write $\lambda=\exp(-2\pi i t)$.

Let $\ga=g\ol \lambda$ with $g$ and $\la$ as above. We denote the fractional part of $x$ by $$\langle x \rangle = x - \lfloor x \rfloor.$$
By definition we have
\begin{align*}a_{g\ol \la}&=\sum_{j=1}^n \langle -tw_j+a_j\rangle
+\sum_{j=1}^r \langle td_i\rangle\\
&= \sum a_j
-\sum_{i=1}^r \lfloor d_it \rfloor
-\sum_{j=1}^n \lfloor a_j- w_jt\rfloor\\
&= \sum a_j
-\sum_{i=1}^r \lfloor d_it \rfloor
+n+\sum_{j=1}^n 
\begin{cases}\lfloor w_jt-a_j\rfloor & w_jt-a_j\not\in \Z\\
\lfloor w_jt-a_j\rfloor -1  & w_jt-a_j\in \Z.
\end{cases}
\end{align*}
The generators of 
$\rm CY$ and $\rm LG$ arising from $\Gamma$ 
can be easily pictured in the following diagram
representing on concentric circles the 
angular coordinates of the elements of $\pmmu_{w_1}\sqcup \dots \sqcup \pmmu_{w_N}$
(black dots) and of $\pmmu_{d_1}\sqcup
\dots \sqcup\pmmu_{d_r}$ (white dots).

\begin{figure}[h]\label{fig:diag1}
\begin{tikzpicture}%
\useasboundingbox (-5.,-5.)--(5.,5.);
\pgfsetroundjoin%
\pgfsetstrokecolor{rgb,1:red,0.7529;green,0.7529;blue,0.7529}
\pgfsetlinewidth{0.2pt}
\pgfsetdash{{5pt}{3pt}}{0pt}
\pgfellipse[stroke]{\pgfxy(0,0)}{\pgfxy(1,0)}{\pgfxy(0,1)}
\pgfellipse[stroke]{\pgfxy(0,0)}{\pgfxy(2,0)}{\pgfxy(0,2)}
\pgfellipse[stroke]{\pgfxy(0,0)}{\pgfxy(3,0)}{\pgfxy(0,3)}
\pgfellipse[stroke]{\pgfxy(0,0)}{\pgfxy(4,0)}{\pgfxy(0,4)}
\pgfellipse[stroke]{\pgfxy(0,0)}{\pgfxy(5,0)}{\pgfxy(0,5)}
\pgfsetstrokecolor{black}
\pgfellipse{\pgfxy(0,0)}{\pgfxy(2.5,0)}{\pgfxy(0,2.5)}
\pgfsetdash{}{0pt}
\pgfxyline(0,0)(5,0)
\pgfxyline(0,0)(-5,0)
\pgfxyline(0,0)(0,5)
\pgfxyline(0,0)(0,-5)
\pgfxyline(0,0)(-2.8868,-5)
\pgfxyline(0,0)(-2.8868,5)
\pgfsetfillcolor{black}
\pgfcircle[stroke]{\pgfxy(1,0)}{.1cm}
\pgfputat{\pgfxy(0.8232,-0.09)}{\pgftext[right,top]{\color{black}\small 1}}\pgfstroke
\pgfcircle[stroke]{\pgfxy(2,0)}{.1cm}
\pgfputat{\pgfxy(1.8232,-0.09)}{\pgftext[right,top]{\color{black}\small 0}}\pgfstroke
\pgfcircle[fillstroke]{\pgfxy(3,0)}{.1cm}
\pgfputat{\pgfxy(2.8232,-0.09)}{\pgftext[right,top]{\color{black}\small 0}}\pgfstroke
\pgfcircle[fillstroke]{\pgfxy(4,0)}{.1cm}
\pgfputat{\pgfxy(3.8232,-0.09)}{\pgftext[right,top]{\color{black}\small 1}}\pgfstroke
\pgfcircle[fillstroke]{\pgfxy(5,0)}{.1cm}
\pgfputat{\pgfxy(4.8232,-0.09)}{\pgftext[right,top]{\color{black}\small 2}}\pgfstroke
\pgfcircle[stroke]{\pgfxy(0,1)}{.1cm}
\pgfputat{\pgfxy(-0.1768,0.8232)}{\pgftext[right,top]{\color{black}\small 2}}\pgfstroke
\pgfcircle[stroke]{\pgfxy(-1,-0)}{.1cm}
\pgfputat{\pgfxy(-1.1768,-0.1768)}{\pgftext[right,top]{\color{black}\small 2}}\pgfstroke
\pgfcircle[stroke]{\pgfxy(-2,-0)}{.1cm}
\pgfputat{\pgfxy(-2.1768,-0.1768)}{\pgftext[right,top]{\color{black}\small 1}}\pgfstroke
\pgfcircle[fillstroke]{\pgfxy(-1.5,-2.5981)}{.1cm}
\pgfputat{\pgfxy(-1.3232,-2.7749)}{\pgftext[left,top]{\color{black}\small 2}}\pgfstroke
\pgfcircle[fillstroke]{\pgfxy(-4,-0)}{.1cm}
\pgfputat{\pgfxy(-4.1768,-0.1768)}{\pgftext[right,top]{\color{black}\small 1}}\pgfstroke
\pgfcircle[stroke]{\pgfxy(-0,-1)}{.1cm}
\pgfputat{\pgfxy(-0.1768,-1.1768)}{\pgftext[right,top]{\color{black}\small 2}}\pgfstroke
\pgfcircle[fillstroke]{\pgfxy(-1.5,2.5981)}{.1cm}
\pgfputat{\pgfxy(-1.6768,2.4213)}{\pgftext[right,top]{\color{black}\small 2}}\pgfstroke
\end{tikzpicture}%
\caption{Diagram of $\{x_1^2+x_2=0,x_1^4+x_2^2+x_3x_1=0\}\subseteq\P(1,2,3);$
$(w_1,w_2,w_3,d_1,d_2)=(1,2,3,-2,-4)$.\label{fig:diag2}}
\end{figure}
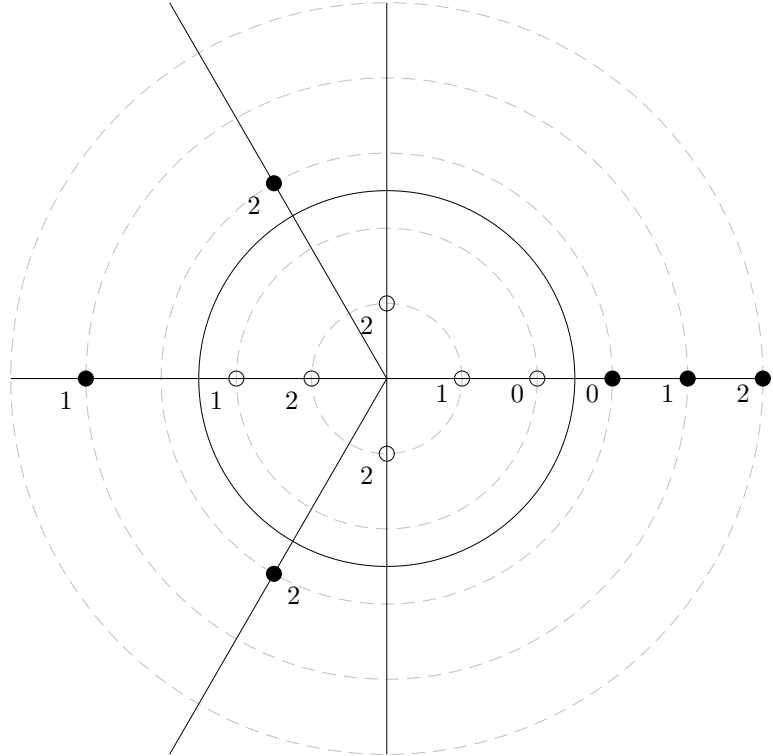 

Let us illustrate how the white and black dots 
correspond to generators of $\rm LG$ and  $\rm CY$, respectively.
Notice that, by construction, a ray $\{\rho\la\mid \rho\in \R^+\}$ 
carries $r_\la$ white dots and $n_\la$ black dots.
The white dots correspond  to the states 
$\LG_\ga^0,\dots,\LG_\ga^{r_\ga-1}$ whereas the black dots correspond to the
states $\CY_\ga^0,\dots,\CY_\ga^{n_\ga-1}$.
The total number of black dots equals the total number of white
dots, \emph{i.e.} $D=\sum_j w_j=\sum_id_i$. 
We can order all the dots in lexicographic order starting from 
$\CY_0$; we get $x_1,x_2,\dots, x_{2D}$. 
Then, by the above formulae, 
the cohomological Chen--Ruan 
degree of a state corresponding to 
$x_i$ is $\sum_j a_j$ plus the following function
$$f(x_i)=
\#\{y \text{ black } \mid x_1\le y<x_i \}-
\#\{y \text{ white } \mid x_1\le y\le x_i \}.$$
We can identify 
each term $x_i$ with the integer $i$ and
extend this degree function to a piecewise linear 
real function defined on 
$]\frac12,2D+\frac12]$
\begin{align*}
f(x)=
\begin{cases} f(x_i)+(x-i) & 
{\rm\ if\ } x \text{ black and\ } x\in]x_i-\frac12,x_i+\frac12],\\
f(x_i)-(x-i) & 
{\rm\ if\ } x \text{ white and\ } x\in]x_i-\frac12,x_i+\frac12].
\end{cases}
\end{align*}
This function is continuous and its values at the boundary of the interval 
coincide, so it can be regarded as a continuous 
function on a circle.
It is strictly increasing at all black dots and strictly decreasing 
at all white dots; its relative maxima occur for 
values of the form $\frac12(x_i+x_{i+1})$ where $x_i$ is black 
and $x_{i+1}$ is white.
Similarly, with inverted colours, for the relative minima.
In this way if $f$ 
attains a given (integer) value at a given number of black dots
it must attain the same value at the same number of
white dots. 
(We illustrate this in Fig.~\ref{fig:diag2} by writing the value of 
$f$ next to each dot). This yields the desired bidegree preserving LG/CY isomorphism.

We can use the function $f$ to 
define an explicit isomorphism, simply by pairing each 
black dot for which $f=k$ to the following white dot 
for which $f=k$. This yields an explicit isomorphism.
\end{proof}

\end{document}